\documentclass[12pt]{article}
\usepackage{amsthm,amsfonts}
\usepackage{tikz}
\usepackage{fullpage}
\usepackage{float}

\newtheorem{thm}{Theorem}
\newtheorem{cor}{Corollary}
\newtheorem{lem}{Lemma}

\date{August 31, 2015}

\begin{document}
\title{Bounds and power means for the general Randi\'c index}
\author{Clive Elphick\thanks{\texttt{clive.elphick@gmail.com}}\quad\quad Pawel Wocjan\thanks{Department of Electrical Engineering and Computer Science, University of Central Florida, Orlando, USA; \texttt{wocjan@eecs.ucf.edu}}}

\maketitle

\abstract{We review bounds for the general Randi\'c index, $R_{\alpha} = \sum_{ij \in E} (d_i d_j)^\alpha$, and use the power mean inequality to prove, for example,  that $R_\alpha \ge m\lambda^{2\alpha}$ for $\alpha < 0$, where $\lambda$ is the spectral radius of a graph. This enables us to strengthen various known lower and upper bounds for $R_\alpha$ and to generalise a non-spectral bound due to Bollob\'as \emph{et al}. We also prove that the zeroth-order general Randi\'c index, $Q_\alpha = \sum_{i \in V} d_i^\alpha \ge n\lambda^\alpha$ for $\alpha < 0$.}

\section{Introduction}

Let $G$ be  a graph with no isolated vertices with vertex set $V(G)$ where $n = |V|$, edge set $E(G)$ where $m = |E|$, degrees $\Delta = d_1 \ge ... \ge d_n = \delta \ge 1$ and average degree $d$. Let $A$ denote the adjacency matrix of $G$ and let $\lambda$ denote the largest eigenvalue of $A$. Let $\omega(G)$ denote the clique number of $G$ and $\chi(G) 
$ denote the chromatic number of $G$. Define the general Randi\'c index, $R_\alpha$, and the zeroth-order general Randi\'c index, $Q_\alpha$,  as usual as:

\[
R_{\alpha} = \sum_{ij \in E} (d_i d_j)^\alpha \mbox{   and  } Q_\alpha = \sum_{i \in V} d_i^\alpha.
\]

$R_{-0.5}$ is the best known and most studied topological index used by mathematical chemists. Gutman \cite{gutman13} published a recent survey of degree-based topological indices, in which he compares the performance of numerous indices in chemical applications. 

Note that $R_\alpha =  M_2^\alpha$ (the variable second Zagreb index) and that $Q_\alpha = M_1^{\alpha/2}$ (the variable first Zagreb index). In particular $R_1 = M_2$ and $Q_2 = M_1$. We will not refer to Zagreb indices again in this paper, and from here onwards $M_p$ will refer to a generalized $p$-mean rather than a Zagreb index. 

In section 2 we introduce the power mean inequality and prove the Lemma which underpins the results in this paper. In sections 3 and 4 we use this Lemma to derive bounds for $R_\alpha$ and $Q_\alpha$ using eigenvalues and degrees respectively. We then review implications of these general bounds for $R_{-1}$ and $R_{-0.5}$ and conclude with a summary of power means for the general Randi\'c indices.

\section{Power mean inequality}

It is convenient to introduce the terminology of power means (also known as generalized means). Let $w_1, ..., w_n$ be $n$ positive real numbers and let $p$ be a real number. Define the sum of the $p$-powers as:

\[
S_p(w_1, ..., w_n) = S_p(w_i) = \sum_{i=1}^n w_i^p
\]

and the generalized $p$-mean, for $p \neq 0$ as:

\[
M_p(w_1, ..., w_n) = M_p(w_i) = \left(\frac{1}{n} S_p(w_1, ..., w_n)\right)^{1/p}.
\]

Throughout this paper we will refer to $M_p(w_i)$ as a $p$-power mean. Note that $p = 1$ corresponds to the arithmetic mean and $p = -1$ corresponds to the harmonic mean. We define $M_0$ to be the geometric mean as follows:

\[
M_0(w_1, ..., w_n) = M_0(w_i) = \left(\prod_{i=1}^n w_i\right)^{1/n}.
\]

It is important that the above definition is consistent with the following limit process:

\[
M_0(w_1, ..., w_n) = \lim_{p\rightarrow 0} M_p(w_1, ..., w_n).
\]
For all real $p < q$ the well known power mean inequality states that:
\begin{equation}\label{eq:power_ineq}
M_p(w_1, ..., w_n) \le M_q(w_1, ..., w_n)
\end{equation}
with equality if and only if $w_1 = ... = w_n$. There are several rigorous proofs of this inequality, including for $p = 0$, for example by Hardy \emph{et al} \cite{hardy52}. 

We use the fact that 
\begin{equation}\label{eq:MpSp}
n \cdot [M_p(w_1,\ldots,w_n)]^p = S_p(w_1,\ldots,w_n)
\end{equation}
for all $w_1,\ldots,w_n>0$ and for all $p$, including the case $p=0$. 

The following lemma is used to prove many of the new bounds in this paper.

\begin{lem}
Let $q$ be arbitrary.  Assume that for the generalized $q$-mean
\[
L \le M_q(w_1, ..., w_n) \le U
\]
where $L$ and $U$ are lower and upper bounds.  Then, we have the following inequalities:

\begin{itemize}
\item
for $p \le q$ and
\begin{itemize}
\item for $p\ge 0$
\[
S_p(w_1, ..., w_n) \le nU^p
\]
\item for $p<0$
\[
S_p(w_1, ..., w_n) \ge nU^p
\]
\end{itemize}
\item
for $p \ge q$ and
\begin{itemize}
\item for $p\ge 0$
\[
n L^p \le S_p(w_1, ..., w_n)
\]
\item for $p<0$
\[
n L^p \ge S_p(w_1, ..., w_n)
\]
\end{itemize}
\end{itemize}
\end{lem} 

\begin{proof}
The power mean inequality (\ref{eq:power_ineq}) implies that $M_p(w_1,\ldots,w_n)\le U$ for $p<q$ and $L \le M_p(w_1,\ldots,w_n)$ for $p>q$.

We apply the function $x \mapsto n x^p$ to go from $M_p(w_1,\ldots,w_n)$ to $S_p(w_1,\ldots,w_n)$ as in (\ref{eq:MpSp}).  We have to reverse the direction of the above inequalities when applying this function for $p<0$.    
\end{proof}

\section{Bounds for $R_\alpha$ and $Q_\alpha$ using eigenvalues}

Favaron \emph{et al} \cite{favaron93} proved that $R_{-0.5} \ge m/\lambda$ and Runge \cite{runge76} and Hofmeister \cite{hofmeister94} proved that $R_{-1} \ge m/\lambda^2$. We can generalise these results as follows.

\begin{thm}
We have the following lower and upper bounds for $R_\alpha$ and $Q_\alpha$:
\begin{itemize}
\item For $\alpha < 0$, 
\[
R_\alpha \ge m\lambda^{2\alpha},
\]
\item For $0 < \alpha \le 0.5$, 
\[
R_\alpha \le m\lambda^{2\alpha},
\]
\item For $\alpha < 0$,  
\[
Q_\alpha \ge n\lambda^\alpha,
\] 
\item For $0 < \alpha \le 1$, 
\[
Q_\alpha \le n\lambda^\alpha.
\]
\end{itemize}
\end{thm}

\begin{proof}

Favaron \emph{et al} \cite{favaron93} proved that:

\[
\left(\frac{1}{m} \sum_{ij \in E} \sqrt{d_i.d_j}\right)^2 \le \lambda^2.
\]

In other words, $\lambda^2$ is an upper bound on the $0.5$-power mean of the $m$ values of $d_i \cdot d_j$ for $(i , j) \in E$. Therefore using Lemma 1 we obtain:
\begin{itemize}
\item For $\alpha < 0$, 
\[
R_\alpha = S_\alpha(d_i \cdot d_j) \ge mU^\alpha = m\lambda^{2\alpha}.
\]
\item For $0 < \alpha \le 0.5$, 
\[
R_\alpha = S_\alpha(d_i \cdot d_j) \le mU^\alpha = m\lambda^{2\alpha}.
\]
\end{itemize}
There is equality in these bounds for $R_\alpha$ when $d_i \cdot d_j$ is equal for all edges in $E$. This is the case for regular graphs and semiregular bipartite graphs.

It is well known that:

\[
\frac{\sum_{i \in V} d_i}{n} = d \le \lambda.
\]

In other words, $\lambda$ is an upper bound on the $1$-power mean of the $n$ values of $d_i$ for $i \in V$. Therefore using Lemma 1 we obtain:
\begin{itemize}
\item For $\alpha < 0$, \[
Q_\alpha = S_\alpha(d_i) \ge nU^\alpha = n\lambda^{\alpha}.
\]
\item For $0 < \alpha \le 1$, \[
Q_\alpha = S_\alpha(d_i) \le nU^\alpha = n\lambda^{\alpha}.
\]
\end{itemize}
There is equality for $Q_\alpha$ when $d_i$ is equal for all vertices in $V$, that is for regular graphs.
\end{proof}

We can derive the following corollaries from Theorem  1 which strengthen known bounds.

Bollob\'as and Erdos \cite{bollobas98} proved that for $-1 \le \alpha < 0$:

\[
R_\alpha \ge m\left(\frac{\sqrt{8m + 1} -1}{2}\right)^{2\alpha}.
\]

We can generalise and strengthen this bound as follows.

\begin{cor}
For $\alpha < 0$, $R_\alpha$ is bounded from below by
\[
R_\alpha \ge m(2m - n + 1)^\alpha. 
\]
\end{cor}

\begin{proof}
Hong \cite{hong93} proved that for graphs with no isolated vertices $\lambda^2 \le (2m - n + 1)$. Therefore using Theorem 1 and that $\alpha < 0$ and that $2m \le n(n - 1)$:

\[
R_\alpha \ge m\lambda^{2\alpha} \ge m(2m - n + 1)^\alpha \ge m\left(\frac{\sqrt{8m + 1} -1}{2}\right)^{2\alpha}. 
\]
\end{proof}

Li and Yang \cite{li04} proved that for $\alpha \le -1$:

\begin{equation}\label{eq:li}
R_{\alpha} \ge \frac{n(n- 1)^{1+2\alpha}}{2}.
\end{equation}

We can strengthen this bound as follows.

\begin{cor}
For $\alpha \le -1$, $R_\alpha$ is bounded from below by
\begin{equation}\label{eq:clive2}
R_\alpha \ge \frac{n^{2\alpha + 2} (\omega - 1)^{2\alpha + 1}}{2\omega^{2\alpha + 1}}.
\end{equation}

\end{cor}

\begin{proof}

Nikiforov \cite{nikiforov02} proved that $\lambda^2 \le 2m(\omega - 1)/\omega$. Noting that $\alpha \le -1$ we have:

\[
R_\alpha \ge m\lambda^{2\alpha} \ge \frac{m(2m(\omega - 1))^\alpha}{\omega^\alpha} = \frac{m^{\alpha+1}2^\alpha(\omega - 1)^\alpha}{\omega^\alpha}.
\]

Turan's theorem states that $m \le n^2(\omega - 1)/2\omega$. Therefore since $\alpha \le -1$:

\[
R_\alpha \ge \frac{n^{2(\alpha + 1)} (\omega - 1)^{\alpha+1} 2^\alpha(\omega - 1)^\alpha}{2^{\alpha+1}\omega^{\alpha+1}\omega^\alpha} = \frac{n^{2\alpha+2}(\omega - 1)^{2\alpha+1}}{2\omega^{2\alpha+1}}.
\]

We can demonstrate that (4) strengthens bound (3) as follows. We wish to show that for $\alpha \le -1$:

\[
\frac{n^{2\alpha + 2} (\omega - 1)^{2\alpha + 1}}{2\omega^{2\alpha + 1}} \ge \frac{n(n- 1)^{1+2\alpha}}{2}.
\]

This simplifies to:

\[
(n(\omega - 1))^{1+2\alpha} \ge ((n - 1)\omega)^{1+2\alpha}.
\]

Take the $(1+2\alpha)$ root of both sides and note that $1+2\alpha \le -1$. Therefore:

\[
n(\omega - 1) \le (n - 1)\omega
\]

which is true for all graphs.

\end{proof}

Lu, Liu and Tian \cite{lu04} proved that for $-1 \le \alpha < 0$:

\[
R_{\alpha} \ge 2^{-\alpha}n^{\alpha}m^{1-\alpha}\lambda^{3\alpha}.
\]

We can generalise this bound as follows.

\begin{cor}
For $\alpha < 0$, $R_\alpha$ is bounded from below by
\[
R_{\alpha} \ge 2^{-\alpha}n^{\alpha}m^{1-\alpha}\lambda^{3\alpha}.
\]

\end{cor}

\begin{proof}
Since $\alpha < 0$ and $\lambda \ge 2m/n$ we have that:

\[
R_{\alpha} \ge m\lambda^{2\alpha} = m\lambda^{3\alpha}\lambda^{-\alpha} \ge m\lambda^{3\alpha}(2m/n)^{-\alpha} = 2^{-\alpha}n^{\alpha}m^{1-\alpha}\lambda^{3\alpha}.
\]

\end{proof}

\section{Bounds for $R_\alpha$ and $Q_\alpha$ using degrees}

Ili\'c and  Stevanovi\'c \cite{ilic09} proved that $R_\alpha \ge md^{2\alpha}$ for $\alpha \ge 0$ and $Q_\alpha \ge nd^\alpha$ for $\alpha \ge 1$. We reproduce and extend these inequalities in the following Theorem, using Lemma 1.

\begin{thm}
We have the following lower and upper bounds on $R_\alpha$ and $Q_\alpha$:
\begin{itemize}
\item For $\alpha \ge 0$,
\[
R_\alpha \ge md^{2 \alpha}.
\]
\item For $\alpha < 0$ and $\alpha \ge 1$, 
\[
Q_\alpha \ge nd^\alpha.
\] 
For $0 < \alpha \le 1$,
\[
Q_\alpha \le nd^\alpha.
\]
\end{itemize}
\end{thm}

\begin{proof}

Ili\'c and Stevanovi\'c \cite{ilic09} proved that:

\[
\frac{R_1}{m} = \frac{\sum_{ij \in E} d_i \cdot d_j}{m} \ge \left(\prod_{ij \in E} d_i \cdot d_j\right)^{1/m} =M_0(d_i \cdot d_j) \ge d^2.
\]

Therefore $d^2$ is a lower bound for the $0$-power mean of $R_\alpha$. Hence using Lemma 1, $R_\alpha \ge md^{2\alpha}$ for $\alpha \ge 0$.

\[
\frac{Q_1}{n} = \frac{2m}{n} = d.
\]

Therefore $d$ can be regarded as a lower and upper bound for the $1$-power mean of $Q_\alpha$. Hence using Lemma 1, $Q_\alpha \ge nd^\alpha$ for $\alpha \ge 1$, $Q_\alpha \le nd^\alpha$ for $0 < \alpha \le 1$ and $Q_\alpha \ge nd^\alpha$ for $\alpha < 0$.
\end{proof}

Bollob\'as and Erd{\"o}s \cite{bollobas98} proved that for $0 < \alpha \le 1$:
\[
R_\alpha \le m\left(\frac{\sqrt{8m + 1} -1}{2}\right)^{2\alpha}.
\]

We strengthen this bound in the following theorem.

\begin{thm}
For $0 < \alpha \le 1$, $R_\alpha$ is bounded from above by 

\[
R_\alpha \le m(2m - n + 1)^\alpha.
\]
\end{thm}

\begin{proof}

Das and Gutman \cite{das04} proved the following bound:

\[
R_1  \le 2m^2 - (n - 1)m\delta + \frac{1}{2}(\delta - 1)m\left(\frac{2m}{n - 1} + n - 2\right).
\]

If $\delta = 1$ then clearly $R_1 \le m(2m - n + 1)$. If $\delta > 1$ then it is straightforward to show that $R_1 \le m(2m - n + 1)$. Therefore $R_1/m \le 2m - n + 1$, so $(2m - n + 1)$ is an upper bound for the $1-$power mean of $R_\alpha$.

Hence using Lemma 1, $R_\alpha \le m(2m - n + 1)^\alpha$ for $0 < \alpha \le 1$.
\end{proof}

\begin{thm}
For $\alpha < 0$, $Q_\alpha$ is bounded from below by 

\[
Q_\alpha \ge n\left(d(\Delta + \delta) - \Delta\delta\right)^{\alpha/2}
\]

\end{thm}

\begin{proof}
Das \cite{das04} proved that:

\[
\frac{Q_2}{n} = \frac{\sum_{i \in V} d_i^2}{n} \le d(\Delta + \delta) - \Delta\delta.
\]

Taking the square root of both sides of this inequality, we see that $\sqrt{d(\Delta + \delta) - \Delta\delta}$ is an upper bound for the  $2$-power mean of $Q_\alpha$. Using Lemma 1 therefore completes the proof.

\end{proof}

\section{Implications for $R_{-1}$}

Cavers \emph{et al} \cite{cavers10} reviewed upper and lower bounds for $R_{-1}$ in the context of bounds for Randi\'c energy. In particular, Shi \cite{shi09} proved that:

\begin{equation}\label{eq:shi}
R_{-1} \ge \frac{n}{2\Delta}
\end{equation}

with equality if and only if $G$ is regular and Li and Yang \cite{li04} proved that:

\begin{equation}\label{eq:li}
R_{-1} \ge \frac{n}{2(n - 1)},
\end{equation}

with equality if and only if $G$ is a complete graph. Liu and Gutman \cite{liu07} proved that for graphs with no isolated vertex:

\begin{equation}\label{eq:liu}
R_{-1} \ge \frac{n - 1}{m}
\end{equation}

with equality only for Star graphs. Clark and Moon \cite{clark00} proved that for trees, $R_{-1} \ge 1$.

Below, in Corollary 4, we prove that:

\begin{equation}\label{eq:clive}
R_{-1} \ge \frac{\omega}{2(\omega - 1)} \mbox{  or equivalently  } \frac{2R_{-1}}{2R_{-1} - 1} \le \omega(G),
\end{equation}

with equality for semiregular bipartite and regular complete $\omega$-partite graphs. (A semiregular bipartite graph is a bipartite graph for which all vertices on the same side of the bipartition have the same degree.)

Bound (8) clearly strengthens bound (6). It also demonstrates that $R_{-1} \ge 1$ not only for trees but for all triangle-free graphs. Bound (8) never outperforms bound (5) for regular graphs but it does outperform bound (5) for some irregular graphs, such as irregular complete bipartite graphs. 

\begin{cor} $R_{-1}$ is bounded from below by
\[
R_{-1} \ge \frac{\omega}{2(\omega - 1)}.
\]

This is exact for semiregular bipartite and regular complete $\omega$-partite graphs.
\end{cor}

\begin{proof}

Letting $\alpha = -1$ we have $R_{-1} \ge m/\lambda^2$. Nikiforov \cite{nikiforov02} proved that:

\[
\lambda^2 \le \frac{2m(\omega - 1)}{\omega}.
\]

Therefore:

\[
R_{-1} \ge \frac{m}{\lambda^2} \ge \frac{m\omega}{2m(\omega - 1)} = \frac{\omega}{2(\omega - 1)}.
\]

\end{proof}

\begin{cor}
For chemical graphs, other than $K_5$, $R_{-1} \ge 2/3$.

\end{cor}

\begin{proof}
For a chemical graph $\Delta \le 4$. It follows from Brooks' famous theorem \cite{brooks41} that, excluding $K_5$, $\omega(G) \le \Delta \le 4$. Therefore:

\[
R_{-1} \ge \frac{\omega}{2(\omega - 1)} \ge \frac{4}{6} = \frac{2}{3}.
\]

This is exact for $K_4$.

\end{proof}

\section{Implications for $R_{-0.5}$}

$R_{-0.5}$ is the original topological index devised by Milan Randi\'c in 1975 and has consequently been investigated more than any other general Randi\'c index.

Bollob\'as and Erd{\"o}s \cite{bollobas98} proved that for graphs with no isolated vertex:

\begin{equation}\label{eq:bela}
\frac{n}{2} \ge R_{-0.5} \ge \sqrt{n - 1}
\end{equation}

with equality only for Star graphs. In Corollary 7 we prove that:

\begin{equation}\label{eq:clive2}
R_{-0.5} \ge \frac{m}{\sqrt{2m - n + 1}}.
\end{equation}

Since connected graphs have $m \ge n - 1$, it is straightforward to show that bound (10)  is never worse than the well known bound (9) for connected graphs.

Hansen and Vukicevi\'c \cite{hansen09} proved that $\chi(G) \le 2R_{-0.5}$. In Corollary 6 we provide a simple alternative proof of this result using Theorem 1.

\begin{cor}

Hansen and Vukicevi\'c \cite{hansen09} proved that $\chi(G) \le 2R_{-0.5}$. We can use Theorem 1 to strengthen their bound as follows.

\[
2R_{-0.5} \ge \lambda + 1 \ge \chi(G).
\]

\end{cor}

\begin{proof}

As noted above $\lambda^2 \le 2m(\omega - 1)/\omega$ and it is well known that $(\omega - 1)\omega \le (\chi - 1)\chi \le 2m$ and that $\chi(G) \le 1 + \lambda.$ Therefore $\lambda \le 2m/\omega$, so:

\[
\lambda(\lambda + 1) \le \frac{2m(\omega - 1)}{\omega} + \frac{2m}{\omega} = 2m.
\]

Hence:

\[
2R_{-0.5} \ge \frac{2m}{\lambda} \ge \lambda + 1 \ge \chi(G).
\]

\end{proof}

\begin{cor}

Hong \cite{hong93} proved that for graphs with no isolated vertices, $\lambda^2 \le (2m - n + 1)$. Therefore with $\alpha = -0.5$:

\[
R_{-0.5} \ge \frac{m}{\lambda} \ge \frac{m}{\sqrt{2m - n + 1}} \ge \sqrt{n - 1} \mbox{ for connected graphs}.
\]

\end{cor}

\section{Summary}

The following tables summarise the power means we have used in this paper.

\begin{table}[H]
\caption{Power means for $R_\alpha$}
\centering
\begin{tabular}{c c c}
\hline \hline
Power mean & Lower bound & Upper bound \\[0.5ex]
\hline
$0.5$ &           & $\lambda^2$ \\
$0$   & $d^2$ &  \\
$1$   &           & $2m - n + 1$ \\
\hline
\end{tabular}
\end{table}

\begin{table}[H]
\caption{Power means for $Q_\alpha$}
\centering
\begin{tabular}{c c c}
\hline \hline
Power mean & Lower bound & Upper bound \\[0.5ex]
\hline
$1$ & & $\lambda$ \\
$1$ & $d$ & $d$ \\
$2$ & & $\sqrt{d(\Delta + \delta) - \Delta\delta}$ \\
\hline
\end{tabular}
\end{table}

There are, we expect, further useful power means for $R_\alpha$ and $Q_\alpha$ to be found.

\section*{Acknowledgements}

We would like to thank Michael Cavers and Tam\'as R\'eti for helpful comments on drafts of this paper.

\end{document}